\documentclass[12pt]{amsart}
\usepackage{epsfig}

\usepackage{amsmath}
\usepackage{amssymb}
\usepackage{amsthm}
\usepackage{epsfig}
\usepackage{graphicx,color}

\newcommand{\RR}{\rm I\kern -1.6pt{\rm R}}

\newtheorem{theorem}{Theorem}[section]
\newtheorem{lemma}{Lemma}[section]

\numberwithin{equation}{section}

\begin{document}
\date{}

\title[Dynamics of epidemic models]{Dynamics of epidemic models with asymptomatic infection and seasonal succession}

\author{\sc Yilei Tang $^{\dag}$ \  Dongmei Xiao $^{\dag*}$
 \ Weinian Zhang $^{\ddag}$ \ Di Zhu $^{\dag}$}

%
%

\thanks{$^*$ Corresponding author.}

\thanks{{\bf Funding}: The first author is partially supported by the National Natural
Science Foundation of China (No. 11431008) and the European Union's Horizon 2020 research and innovation
programme under the Marie Sklodowska-Curie grant agreement (No. 655212). The second author is supported by  the National Natural
Science Foundation of China (No. 11431008 \& 11371248). The third author is supported by  the National Natural
Science Foundation of China (No. 11521061 \&  11231001). }

\thanks{$^{\dag}$ School of Mathematical Science, Shanghai Jiao Tong University,  Shanghai, 200240, P. R. China
(xiaodm@sjtu.edu.cn (D. Xiao),  mathtyl@sjtu.edu.cn (Y. Tang), di.zhu@auckland.ac.nz (D. Zhu)) }

 \thanks{$^{\ddag}$  Yangtze Center of Mathematics and Department of Mathematics, Sichuan University, Chengdu, Sichuan 610064, P. R. China  (matzwn@163.com (W. Zhang))}



\maketitle

{\bf Abstract~~} In this paper, we consider a compartmental SIRS epidemic model with asymptomatic infection and seasonal succession,
which is a periodic discontinuous differential system.
 The basic reproduction number $\mathcal{R}_0$ is defined and evaluated directly for this model, and the uniformly persistent of the disease
 and threshold dynamics are obtained. Specially, global dynamics
 of the model without seasonal force are studied. It is shown that
the model has only a disease-free equilibrium which is globally stable if $\mathcal{R}_0\le 1$,
and as $\mathcal{R}_0>1$  the disease-free equilibrium is unstable and the model has an endemic equilibrium, which is globally stable
if the recovering rates of asymptomatic infective and symptomatic infective are close. These theoretical results  provide an intuitive basis for
  understanding that the asymptomatic infective individuals and the disease seasonal transmission promote  the evolution of epidemic,
which allow us to predict the outcomes  of control strategies during the course of the epidemic.

\section{Introduction}

Since Kermack and McKendrick in \cite{Ker-McK} proposed the classical deterministic compartmental
model (called SIR model) to describe epidemic  outbreaks and spread,
mathematical models have become important tools in analyzing the
spread and control of infectious diseases, see \cite{AleM2005, May, Brauer, SeasonalScience, Heth2000, Hsi2014, THRZ08, Tow2012,  xiaoruan}
and references therein. The number of infected individuals used in these models is usually calculated via data in the hospitals. However,
some studies on influenza show that  some individuals of the population who
are infected never develop symptoms, i.e. being asymptomatic infective.  The asymptomatic infected individuals will not go to hospital but they can infect the susceptible by contact, then go to the recovered stage, see for instance
\cite{wujh, Long, Feng}. Hence, using the data from hospitals to mathematical models to assess the epidemic,
it seems that we will underestimate infection risks in epidemic.

On the other hand, seasonality is very common in ecological and human social systems (cf. \cite{xiao}). For example, variation patterns in climate are repeated every year,
birds migrate according to the variation of season, opening and closing of schools are almost periodic, and so on. These seasonal factors significantly influence
the survival of pathogens in the environment, host behavior, and abundance of vectors and non-human hosts.
A number of papers have suggested that seasonality plays
an important role in epidemic outbreaks and the evolution of disease transmissions, see
\cite{PeriodWu, periodC, SeasonalD, SeasonalScience, seasonalS, Smi1983, seasonalNature, Tow2012, zhang2007}. However, it is still
challenging to understand the mechanisms of seasonality and
their impacts on the dynamics of infectious diseases.

Motivated by the studies above on asymptomatic infectivity or seasonality, we develop a compartmental model with asymptomatic infectivity and seasonal factors in this paper.
This model is a periodic discontinuous differential system.
We try to establish the  theoretical analysis on the periodic discontinuous differential systems and obtain the dynamics of the model. This will allow us to draw both qualitative and quantitative conclusions on effect of the asymptomatic infectivity and seasonality on the epidemic.

 The rest of the paper is organized as follows.
 In section 2, we formulate the SIRS model with asymptomatic infective and seasonal factors, then discuss the existence and regularity of non-negative solutions
 for this model.
 In section 3, we define the basic reproduction number $\mathcal{R}_0$   for the model, give the evaluation of  $\mathcal{R}_0$
 and investigate the threshold dynamics  of the model (or the uniformly persistent of the disease). It is shown that the length of the season, the transmission rate
and the existence of asymptomatic infective affect the basic reproduction number  $\mathcal{R}_0$. In section 4, we study the global dynamics of the model ignoring seasonal factor.
We prove that there is a unique disease-free equilibrium and the disease always dies out when $\mathcal{R}_0\le 1$;
 while when $\mathcal{R}_0> 1$ there is an endemic equilibrium which is global stable  if the recovering rates of asymptomatic infective and symptomatic infective are close.
A brief discussion is given in the last section.

\section{Model formulation}
In this section, we first extend the classic SIRS model  to a model which incorporates with
the asymptomatic infective  and seasonal features of epidemics,  and then study the regularity of solutions of the model.

Because there are asymptomatic infectious and
symptomatic infectious individuals in the evolution of epidemic, the whole population is divided into four compartments: susceptible, asymptomatic infectious,
symptomatic infectious and recovered individuals.
More precisely, we let $S$, $I_a$, $I_s$ and $R$ denote the numbers of
individuals in the susceptible, asymptomatic,
symptomatic and recovered compartments, respectively,
and $N$ be the total population size.  Let $\mathbb{R}_+=[0, +\infty)$, $\mathbb{Z}_+$ be the set of all nonnegative integers, and $\omega>0$ be given as the period of the disease transmissions. In addition to the assumptions of the classical SIRS model, we list the following assumptions on seasonal factors,
 asymptomatic infectivity and symptomatic infectivity.

\begin{itemize}
\item[(A1)] Due to the opening and closing of schools or migration of birds, each period of the disease transmission is simply divided into two seasons with high and low
transmission rates, which are called high season $J_2$ and low season $J_1$, respectively. The seasonality is described by a piecewise constant function with high transmission rate $\beta_2$ in $J_2$ and low transmission rate $\beta_1$ in $J_1$, respectively, where $J_1=[m\omega, m\omega+(1-\theta)\omega )$ and $J_2=[ m\omega+(1-\theta)\omega, (m+1)\omega)$.
Here $m\in \mathbb{Z}_+$,  and  $0<\theta<1$  which measures the fraction of the high season to the whole infection cycle.
\item[(A2)] There are two classes of infective individuals: asymptomatic infective ones and symptomatic infective ones. Both of them are able to infect susceptible individuals by contact.
    A fraction $\mu$ of infective individuals  proceeds to the
asymptomatic infective compartment and the remainder (i.e.  a fraction $1-\mu$ of infective individuals) goes directly to the symptomatic infective compartment. And the asymptomatic infective and symptomatic infective individuals recover from disease at rate $r_a$ and $r_s$, respectively.
\item[(A3)]  The symptomatic infective individuals will get treatment in hospital or be quarantined. Hence, the symptomatic infective individuals reduce their contact
rate by a fraction $\alpha$.
\end{itemize}

Based on these assumptions, the classical SIRS model can be extended to the following system

\begin{equation}\label{model}
\begin{cases}
\dot{S}(t)=dN(t)-dS(t)-\beta(t) S(t)(I_{a}(t)+\alpha I_{s}(t))+\sigma R(t), \\
\dot{I_{a}}(t)=\mu\beta(t) S(t)(I_{a}(t)+\alpha I_{s}(t))-(d+r_{a})I_{a}(t), \\
\dot{I_{s}}(t)=(1-\mu)\beta(t) S(t)(I_{a}(t)+\alpha I_{s}(t))-(d+r_{s})I_{s}(t), \\
\dot{R}(t)=r_{a}I_{a}(t)+r_{s}I_{s}(t)-(d+\sigma)R(t),
\end{cases}
\end{equation}
where $N(t)=S(t)+I_a(t)+I_s(t)+R(t)$, all parameters $d$, $\alpha$, $\sigma$, $\mu$, $r_a$ and $r_s$ are nonnegative, and
$$
\beta(t)=\left\{ \begin{array}{ll}
\beta_1, \ & t\in J_1=[m\omega, m\omega+(1-\theta)\omega ),\\
\beta_2, \ & t\in J_2=[ m\omega+(1-\theta)\omega, (m+1)\omega ).\\
\end{array}
\right.
$$
 Parameters $\beta_2$  and $\beta_1$ are the rates of contact transmission of the disease in high season and low season respectively for which $\beta_2\ge \beta_1\ge 0$.
 Besides,  $d$ is  both birth rate and death rate, $\alpha$, $0\le\alpha\le 1$, is the fraction of the symptomatic infective individuals reducing their contact
 rate with susceptible,  the fraction of infective individuals becoming asymptomatic infective $\mu$ satisfies $0\le\mu\le 1$,
  parameter $\sigma$ is the rate of recovered population loss of the immunity and  reentering the susceptible group,
  and $r_a$ and $r_s$ are the rates of asymptomatic infective and symptomatic infective recovering with immunity, respectively.

From the biological point of view, we focus on the solutions of system (\ref{model}) with initial conditions
\begin{equation}\label{initial}
 S(0)=S_0 \ge 0, I_a(0)=I_{a0} \ge 0, I_s(0)=I_{s0}\ge 0, R(0)=R_0 \ge 0
 \end{equation}
in the first octant $\mathbb{R}_+^4$.

Note that
\[
\dot{ N}(t)=\dot{S}(t) + \dot{ I}_a(t) + \dot{I}_s(t) + \dot{ R}(t) \equiv 0, \ t\in J_1 \ \rm{or}\ t\in J_2.
\]
 Hence,  $N(t)=S_0+I_{a0}+I_{s0}+R_0$, which is a constant for almost all $t\in \mathbb{R}_+$.
Since the total population does not change by the assumption, we let
\[
S(t) + I_a(t) +I_s(t) + R(t) \equiv N
\]
for almost all $t\in \mathbb{R}_+$. Therefore,  system \eqref{model} with the initial condition \eqref{initial} in $\mathbb{R}_+^4$ can be reduced to
\begin{equation}\label{SIRS3}
\begin{cases}
\dot{S}=(d+\sigma)(N-S)-\beta (t) S(I_a+\alpha I_s)-\sigma (I_a+I_s), \\
\dot{I_a}=\mu\beta (t) S(I_a+\alpha I_s)-(d+r_a)I_a, \\
\dot{I_s}=(1-\mu)\beta (t) S(I_a+\alpha I_s)-(d+r_s)I_s,  \\
S(0)=S_0, I_a(0)=I_{a0}, I_s(0)=I_{s0},\\
P_0=(S_0, I_{a0}, I_{s0})\in \mathcal{D}_0,
\end{cases}
\end{equation}
where $\mathcal{D}_0\subset \mathbb{R}_+^3$ and
\begin{equation}\label{D}
\mathcal{D}_0:=\{(S, I_a, I_s)|\; S\ge 0,  I_a\ge 0, I_s\ge 0,  ~0\le S+ I_a+I_s\le N \}.
\end{equation}
Clearly, the right hand side of system \eqref{SIRS3} is not continuous on the
domain $\mathbb{R}_+\times\mathcal{D}_0$. We  claim that the solution of system \eqref{SIRS3} exists globally on the interval $\mathbb{R}_+=[0, +\infty)$ and is unique.

\begin{theorem}\label{existenceUni}
For any $P_0\in \mathcal{D}_0$, system \eqref{SIRS3} has a unique global solution $\varphi(t, P_0)=(S(t,P_0), I_a(t,P_0), I_s(t,P_0))$
 in $\mathbb{R}_+$, which is continuous with respect to $t$ and all parameters of this system.

Moreover, $\varphi(t, P_0)\subseteq \mathcal{D}_0$ for any $t\in \mathbb{R}_+$ and the solution $\varphi(t, P_0)$ is differentiable with respect to $P_0$, where
 some derivatives are one-sided if $P_0$ is on the domain boundary.
\end{theorem}

\begin{proof}
Assume that $\varphi(t,P_0)$ is a solution of system \eqref{SIRS3}. We first
consider the two systems
\begin{equation}\label{SIR-dim32}
\begin{cases}
\dot{S}=(d+\sigma)(N-S)-\beta_i S(I_a+\alpha I_s)-\sigma (I_a+I_s),\\
\dot{I}_a=\mu\beta_i S(I_a+\alpha I_s)-(d+r_a)I_a,\\
\dot{I}_s=(1-\mu)\beta_i S(I_a+\alpha I_s)-(d+r_s)I_s,\\
S(t_*)=S_*,\; I_{a}(t_*)=I_{a*},\; I_{s}(t_*)=I_{s*},\\
P_*=(S_*, I_{a*},I_{s*})\in \mathbb{R}_+^3
\end{cases}
\end{equation}
in the domain $\mathbb{R}_+\times\mathbb{R}_+^3$, $i=1, 2,$ respectively.

It is clear that for each $i$ the solution of system \eqref{SIR-dim32} exists and is unique on its maximal interval of existence, and the solution of system \eqref{SIR-dim32} is differentiable with respect to the initial value $P_*$ by the fundamental theory of ordinary differential equations.

Note that
the bounded closed set $\mathcal{D}_0$ in $\mathbb{R}_+^3$ is a positive compact invariant set of system \eqref{SIR-dim32} since  the vector field of system \eqref{SIR-dim32} on
the boundary $\partial\mathcal{D}_0$ of $\mathcal{D}_0$ is directed toward to the interior of $\mathcal{D}_0$ or lies on $\partial\mathcal{D}_0$, where
\begin{eqnarray*}
\begin{split}
\partial\mathcal{D}_0=&\{(S,I_a,I_s):\ (S,I_a,I_s)\in \mathbb{R}_+^3, S=0, \ 0\le I_a+I_s\le N\}\\
&\cup\{(S,I_a,I_s):\ (S,I_a,I_s)\in \mathbb{R}_+^3, I_s =0,
\ 0\le S+I_a\le N\}  \\
 & \cup\{(S,I_a,I_s): \ (S,I_a,I_s)\in \mathbb{R}_+^3, I_a=0,\  0\le S+I_s\le N\}\\
 &\cup \{(S,I_a,I_s): \ (S,I_a,I_s)\in \mathbb{R}_+^3, S+I_s+I_a=N\}.
\end{split}
\end{eqnarray*}
Therefore,  the solution of system \eqref{SIR-dim32} exists globally  for any $P_*\in \mathcal{D}_0\subset\mathbb{R}_+^3$, and these solutions are in $\mathcal{D}_0$
for all $t> 0$.

Let $\phi_i(t,t_*, P_*)$ for $i=1, 2$
be the solution semiflow of the following system
\begin{equation}\label{SIRS1}
\begin{cases}
\dot{S}=(d+\sigma)(N-S)-\beta_i S(I_a+\alpha I_s)-\sigma (I_a+I_s),\\
\dot{I}_a=\mu\beta_i S(I_a+\alpha I_s)-(d+r_a)I_a,\\
\dot{I}_s=(1-\mu)\beta_i S(I_a+\alpha I_s)-(d+r_s)I_s,\\
\phi_i(t_*,t_*, P_*)=P_*, \ P_*\in \mathcal{D}_0,
\end{cases}
\end{equation}
respectively,  that is,  $\phi_i(t,t_*, P_*) = (S(t,t_*, P_*), I_a(t,t_*, P_*), I_s(t,t_*, P_*))$ for $t\ge t_*$  is the
solution of system \eqref{SIRS1} with the initial condition $\phi_i(t_*,t_*, P_*)=(S_*, I_{a*},I_{s*})\in \mathcal{D}_0$, respectively.

It follows that the solution
$\varphi(t,P_0)$ for  $t\ge 0$  of system \eqref{SIRS3} can be determined uniquely  by induction.
For simplicity, we let $s_m=(m-1)\omega$ and  $t_m=s_m+(1-\theta)\omega$ for $m\in \mathbb{Z}_+$.
Hence,
$$
[0,\infty )=\bigcup _{m=1}^\infty [s_m, s_{m+1}]
=\bigcup _{m=1}^\infty
([s_{m}, t_m]\cup [t_m, s_{m+1}]),
$$
and $\varphi(t,P_0)$ can be written as follows.
\begin{equation}\label{solution}
\varphi(t,P_0)=\left\{
  \begin{array}{ll}
   \phi _1(t,s_1, P_0) &\textrm{when}\;t\in [s_1, t_1],
   \\[1ex]
  \phi _2(t, t_1, \phi_1(t_1, s_1, P_0)) &\textrm{when}\;\;t\in [t_1, s_2], \\[1ex]
  ...
  \\[1ex]
  \phi _1(t, s_m,  u_m) & \textrm{when}\;\;t\in [s_m , t_m],  \\[1ex]
  \phi _2(t, t_m,  v_m) & \textrm{when}\;\;t\in [t_m, s_{m+1} ],
    \end{array}
\right.
\end{equation}
where $u_m$ and $v_m$ are determined by letting $u_1=P_0$, $v_1=\phi_1(t_1, s_1, u_1)$ and
$$
u_m=\phi _2(s_m, t_{m-1}, v_{m-1}),\; v_m=\phi _1(t_m, s_m, u_m) \;\;\textrm{for}\;\;m\geq 2.
$$
 This implies that the solution $\varphi(t,P_0)$ of system \eqref{SIRS3} exits globally in $\mathbb{R}_+$ and is unique for any $P_0\in \mathcal{D}_0$,
 and  it is continuous with respect to $t$ and all parameters.

 By the expression \eqref{solution},
it is easy to see that the solution $\varphi(t,P_0)$ lies in $\mathcal{D}_0$ for all $t\ge 0$ and $\varphi(t,P_0)$ is differentiable with respect to $P_0$.
The proof is completed.
\end{proof}
Theorem \ref{existenceUni} tells us that  system \eqref{SIRS3} is $\omega$-periodic with respect to $t$ in $\mathbb{R}_+\times \mathcal{D}_0$, and it suffices to investigate the dynamics of its associated period map $\mathcal{P}$ on $\mathcal{D}_0$ for the dynamics of system \eqref{SIRS3}, where
\begin{equation}\label{poincaremap}
\begin{split}
\mathcal{P}:& \ \mathcal{D}_0 \to \mathcal{D}_0,\\
 \mathcal{P}(P_0)&=\varphi(\omega,P_0)=\phi _2(\omega, (1-\theta)\omega, \phi_1((1-\theta)\omega, 0, P_0)),
 \end{split}
\end{equation}
which is continuous in $\mathcal{D}_0$.


\section{ Basic reproduction number and threshold dynamics}

In epidemiology, the basic reproduction number (or basic reproduction ratio) $\mathcal{R}_0$ is an important quantity,
defined as the average number of secondary infections produced when an infected individual is introduced into a host population where everyone is susceptible. It is often considered as the threshold quantity that determines whether
an infection can invade a new host population and persist. Detailedly speaking, if $\mathcal{R}_0<1$, the disease dies out
and the disease cannot invade the population; but if $\mathcal{R}_0>1$, then the disease is established in the population. There have been some successful approaches for
the calculations of basic reproduction number for different epidemic models. For example, Diekmann {\it et al} in \cite{Die1990}
and van den Driessche and Watmough in \cite{Van2002} presented general approaches of  $\mathcal{R}_0$ for autonomous continuous epidemic models.  And for periodic continuous epidemic models,
Wang and Zhao in \cite{Wang2008}
defined the basic reproduction number. 
 Under some  assumptions on the discontinuous states function,
Guo,  Huang and Zou \cite{Guo2012} determine the basic reproduction number for  an SIR epidemic model with discontinuous treatment strategies.
To our knowledge, there is no theoretic approach to calculate the basic reproduction number for periodic discontinuous epidemic models such as system \eqref{SIRS3}.
In this section, we use the idea and some notations given in \cite{Wang2008} to define and calculate the basic reproduction numbers for system \eqref{SIRS3},
and discuss the uniformly persistent of the disease  and threshold dynamics.


We define ${\bf X}$ to be the set of all disease free states of system \eqref{SIRS3}, that is
$${\bf X}=\{(S,I_a,I_s):\ 0\le S\le N, I_a=I_s=0\}.$$
Clearly,  the disease free subspace ${\bf X}$ is positive invariant for system \eqref{SIRS3}.
 It can be checked that the period map $\mathcal{P}(P_0)$ in ${\bf X}$ has a unique fixed point at $(N,0,0)$, which is a unique disease-free equilibrium $(N,0,0)$
of  system  \eqref{SIRS3}, denoted by $E_0$. We now consider a population near the disease-free equilibrium $E_0$.

For simplicity,  we let $\mathbf{x}=(S,I_a,I_s)^T$, and for $i=1,2$ set
 \begin{equation*}\label{FV}
 \begin{array}{ll}
 \ \mathbf{F_i}&=\left(
  \begin{array}{rrr}
0  & 0  & 0
\\
0 & \mu\beta_iN & \alpha\mu\beta_iN
\\
0 & (1-\mu)\beta_iN &  \alpha(1-\mu)\beta_iN
  \end{array}
 \right):=\left(
  \begin{array}{rr}
0  & {\bf 0 }
\\
{\bf 0} & F_i
  \end{array}
 \right),\\ \mathbf{V_i}&=\left(
  \begin{array}{rrr}
d+\sigma  & \beta_iN+\sigma  & \alpha\beta_iN+\sigma
\\
0 & d+r_a & 0
\\
0 & 0 &  d+r_s
  \end{array}
 \right):=\left(
  \begin{array}{rr}
d+\sigma  & {\bf b_i }
\\
{\bf 0} & V
  \end{array}
 \right).
 \end{array}
 \end{equation*}

Then the linearized system of \eqref{SIRS3} at $E_0$ can be rewritten as
\begin{equation}\label{real_linear}
\frac{d\mathbf{x}}{dt}=(\mathbf{F}(t)-\mathbf{V}(t))\mathbf{x},
\end{equation}
 where $\mathbf{F}(t)=\chi_{J_1}(t)\mathbf{F}_{1}\ +\chi_{J_2}(t)\mathbf{F}_{2}\ $, $\mathbf{V}(t)=\chi_{J_1}(t)\mathbf{V}_{1}+\chi_{J_2}(t)\mathbf{V}_{2}$, and
 $$
 \chi_{J_i}(t)=\left\{\begin{array}{ll}
 1  & \ {\textrm as }\ t\in J_i,
\\
0 & \ {\textrm as }\ t\notin J_i.
  \end{array}
 \right.
 $$
System \eqref{real_linear} is a piecewise continuous periodic linear system with period $\omega$ in $t\in \mathbb{R}_+$.
In order to determine the fate of a small number of infective individuals introduced into a disease free population,
we first extend system \eqref{real_linear} from $t\in \mathbb{R}_+$ to $t\in \mathbb{R}$, and introduce some new notations.
When $t\in \cup_{m=-\infty}^{+\infty}(J_1\cup J_2)=(-\infty, +\infty),$ we set $\mathbb{I}(t)=(I_a(t),I_s(t))^T$, and  $$
\mathbb{F}(t)=
\chi_{J_1}(t){F}_{1}\ +\chi_{J_2}(t){F}_{2}=\left(
  \begin{array}{rr}
 \mu N\beta(t) & \alpha\mu  N\beta(t)
\\
 (1-\mu) N\beta(t)  &  \alpha(1-\mu)  N\beta(t)
  \end{array}
 \right),
$$
 where
 \[ \beta(t)=\left\{ \begin{array}{ll}
\beta_1, \ & t\in J_1=[m\omega, m\omega+(1-\theta)\omega ),\\
\beta_2, \ & t\in J_2=[ m\omega+(1-\theta)\omega, (m+1)\omega), \ m\in \mathbb{Z}.\\
\end{array}
\right.\]
Clearly,
$\mathbb{F}(t)$ is a $2\times 2$ piecewise  continuous periodic matrix with period $\omega$ in $\mathbb{R}$, and it is non-negative. And
$$
-V=\left(
  \begin{array}{rr}
 -(d+r_a) & 0
\\
 0  &  -(d+r_s)
  \end{array}
 \right),
$$
which is cooperative in the sense that the off-diagonal
 elements of $-V$ are non-negative.

Let $Y(t,s)$, $t\ge s$, be the evolution operator of the linear system
\begin{equation}\label{vv}
\frac{d\mathbb{I}(t)}{dt}=-V\mathbb{I}(t).
\end{equation}
Since $V$ is  a constant matrix,  for each $s\in \mathbb{R}$ the matrix $Y(t,s)$ satisfies
\begin{equation}\label{v2eq}
\frac{d}{dt}Y(t,s)=-V Y(t,s), \ t\ge s, \ Y(s,s)=E^2,
\end{equation}
where $E^2$ is a $2\times 2$ identity matrix, and $Y(t,s)=e^{-V(t-s)}$. Hence, the monodromy matrix $\Phi_{-V}(t)$ of system \eqref{vv} is $Y(t,0)$, that is,
$$
\Phi_{-V}(t)=e^{-Vt}=\left(
  \begin{array}{ll}
e^{-(d+r_a)t} & 0
\\
 0 & e^{-(d+r_s)t}
  \end{array}
 \right),
$$
 where $d$, $r_a$ and $r_s$ are positive numbers.

We denote $\|\cdot\|_1$ the $1$-norm of vector and matrix.
Thus, there exist $K>0$ and $\kappa >0$ such that
$$
\|Y(t,s)\|_1\le Ke^{-\kappa (t-s)}, \ \forall t\ge s, \ s\in \mathbb{R}.
$$
And from the boundedness  of $\mathbb{F}(t)$, i.e. $\|\mathbb{F}(t)\|_1<K_1$, it follows that there exists  a constant $K_1>0$ such that
\begin{equation}\label{Lineq}
\|Y(t,t-a)\mathbb{F}(t-a)\|_1\le K K_1e^{-\kappa a},\ \forall t\in \mathbb{R},\ a\in [0, +\infty).
\end{equation}

 We now consider the distribution of infected individuals in the periodic environment. Assume that $\mathbb{I}(s)$ is the initial distribution of infected individuals in infectious compartments. Then  $\mathbb{F}(s)\mathbb{I}(s)$
is the distribution of new infections produced by the infected individuals who were introduced at time $s$.
  Given $t\ge s$, then $Y(t,s)\mathbb{F}(s)\mathbb{I}(s)$
  is the distribution of
those infected individuals  which were newly infected at time s and still remain in the infected
compartments at time t.
Thus, the integration of this distribution from $-\infty$ to $t$
$$
\int_{-\infty}^tY(t,s)\mathbb{F}(s)\mathbb{I}(s)ds
=\int_0^{\infty}Y(t,t-a)\mathbb{F}(t-a)\mathbb{I}(t-a)da
$$
gives the distribution of cumulative new infections at time t produced by all those
infected individuals introduced at times earlier than $t$.

Let $\mathbb{C}_{\omega}=\mathbb{C}(\mathbb{R},\mathbb{R}^2)$  be the ordered Banach space of $\omega$-periodic  continuous  functions from $\mathbb{R}$ to $\mathbb{R}^2$, which is equipped with the  norm $\left \lVert \cdot \right \rVert _c$,
$$
\left \lVert \mathbb{I}(s) \right \rVert_c=\max _{s\in [0,\omega ]}\left \lVert\mathbb{I}(s)\right \rVert_1,
$$
 and the generating positive cone
 $$
 \mathbb{C}^+_{\omega}=\{\mathbb{I}(s)\in \mathbb{C}_{\omega}:\ \mathbb{I}(s)\ge 0, \ s\in \mathbb{R}\}.
 $$

Define a linear
operator $\mathcal{L}: \ \mathbb{C}_{\omega}\to \mathbb{C}_{\omega}$ by
\begin{equation}\label{operate}
(\mathcal{L}\mathbb{I})(t)=\int_{-\infty}^tY(t,s)\mathbb{F}(s)\mathbb{I}(s)ds=\int_0^{\infty}Y(t,t-a)\mathbb{F}(t-a)\mathbb{I}(t-a)da.
\end{equation}
It can be checked that the linear operator $\mathcal{L}$ is well defined.

\begin{lemma}\label{operatorC}
The operator $\mathcal{L}$ is positive, continuous and compact on $\mathbb{C}_{\omega}$.
\end{lemma}
\begin{proof}
Since $Y(t,s)=e^{-V(t-s)}$ and $\mathbb{F}(t)$ is a nonnegative bounded matrix,
we get that $\mathcal{L}(\mathbb{C}^+_{\omega})\subset \mathbb{C}^+_{\omega}$. This implies that the
linear operator $\mathcal{L}$ is positive.

We now prove the continuity of $\mathcal{L}$. For each $t\in  \mathbb{R}$, we have
\begin{eqnarray*}
\begin{split}
\|\mathcal{L}\mathbb{I}(t)\|_{1}&=\left \|\int_0^{\infty}Y(t,t-a)\mathbb{F}(t-a)\mathbb{I}(t-a)da\right \|_{1}\\
&=\left \|\sum_{j=0}^\infty\int_{j\omega}^{(j+1)\omega}Y(t,t-a)\mathbb{F}(t-a)\mathbb{I}(t-a)da\right \|_{1}\\
&\leq \sum_{j=0}^\infty\int_{j\omega}^{(j+1)\omega}\|Y(t,t-a)\mathbb{F}(t-a)\mathbb{I}(t-a)\|_{1}da\\
&\leq \sum_{j=0}^\infty\int_{j\omega}^{(j+1)
\omega}KK_1e^{-\kappa a}\|\mathbb{I}(t-a)\|_1da\\
&\leq \omega K K_1\sum_{j=0}^\infty e^{-\kappa\omega j}\cdot\|\mathbb{I}\|_c
\end{split}
\end{eqnarray*}
by \eqref{Lineq}. Hence,
\begin{equation*}
  \|\mathcal{L}\mathbb{I}(t)\|_{c}=\max_{t\in[0,\omega]}\|\mathcal{L}\mathbb{I}(t)\|_{1}\leq \omega K K_1\sum_{j=0}^\infty e^{-\kappa\omega j}\cdot\|\mathbb{I}\|_c,
\end{equation*}
which implies that $\mathcal{L}$ is continuous and uniformly bounded since $\sum_{j=0}^\infty e^{-\kappa\omega j}$ is convergent.

In the following we prove the compactness of $\mathcal{L}$. We first claim that $\mathcal{L}\mathbb{I}(t)$ is equicontinuous.
Consider $\mathbb{I}(t)\in \mathbb{C}_{\omega}$  and $\forall t_1,t_2\in [0,\omega]$ with $t_1<t_2$. Then
\begin{eqnarray*}
\begin{split}
&\|\mathcal{L}\mathbb{I}(t_2)-\mathcal{L}\mathbb{I}(t_1)\|_1=\left\|\int_{-\infty}^{t_2}Y(t_2,s)\mathbb{F}(s)\mathbb{I}(s)ds-
\int_{-\infty}^{t_1}Y(t_1,s)\mathbb{F}(s)\mathbb{I}(s)ds\right \|_1\\
&=\left \|\int_{-\infty}^{t_2}(Y(t_2,s)-Y(t_1,s))\mathbb{F}(s)\mathbb{I}(s)ds
+\int_{t_1}^{t_2}Y(t_1,s)\mathbb{F}(s)\mathbb{I}(s)ds\right \|_1\\
&\leq \int_{-\infty}^{t_2}\|Y(t_2,s)-Y(t_1,s)\|_1\|\mathbb{F}(s)\|_1\|\mathbb{I}(s)\|_1ds+\int_{t_1}^{t_2}\|Y(t_1,s)\|_1\|\mathbb{F}(s)\|_1\|\mathbb{I}(s)\|_1ds\\
&\leq \int_{-\infty}^{\omega}\|Y(t_2,s)-Y(t_1,s)\|_1\|\mathbb{F}(s)\|_1\|\mathbb{I}(s)\|_1ds+\int_{t_1}^{t_2}Ke^{-\kappa (t_1-s)}
\|\mathbb{F}(s)\|_1\|\mathbb{I}(s)\|_1ds\\
&\leq\|e^{-Vt_2}-e^{-Vt_1}\|_1\sum_{i=-\infty}^0\int_{i\omega}^{(i+1)\omega}K_1\|e^{Vs}\|_1\|\mathbb{I}(s)\|_1ds+\int_{t_1}^{t_2}Ke^{-\kappa (t_1-s)}
K_1\|\mathbb{I}(s)\|_1ds\\
&\leq\sum_{i=-\infty}^0e^{\tilde{d}_1(i+1)\omega}\cdot K_1\|\mathbb{I}\|_c\|e^{-Vt_2}-e^{-Vt_1}\|_1+KK_1e^{\kappa \omega }\|\mathbb{I}\|_c(t_2-t_1),
\end{split}
\end{eqnarray*}
where $\tilde{d}_1=\max\{d+r_a, d+r_s\}$.

Notice that $\sum_{i=-\infty}^0e^{N(i+1)}$ is convergent and $e^{-Vt}$ is continuous on $[0,\omega]$.
Thus, if $\{\mathbb{I}(t)\}$ is bounded,  for $\forall \epsilon>0$ there exists a $\delta>0$ such that $\|\mathcal{L}\mathbb{I}(t_2)-\mathcal{L}\mathbb{I}(t_1)\|_c<\epsilon$ as $|t_2-t_1|<\delta$. This implies that $\{(\mathcal{L}\mathbb{I})(t)\}$ are equicontinuous. According to Ascoli-Arzela theorem, we know that $\mathcal{L}$ is compact. The proof of this lemma is completed.
\end{proof}

$\mathcal{L}$ is called the next infection operator, and the spectral radius of $\mathcal{L}$ can be defined as the
basic reproduction number (or ratio)
\begin{equation}\label{R_0}
\mathcal{R}_0:= \rho(\mathcal{L})
\end{equation}
of system \eqref{SIRS3}.

Following \cite{Wang2008}, we consider how to calculate $\mathcal{R}_0$ and whether the basic reproduction ratio (or number)
$\mathcal{R}_0$ characterizes the threshold of disease invasion, i.e., the disease-free periodic
solution $(N,0,0)$ of system \eqref{SIRS3} is local asymptotically stable if $\mathcal{R}_0 < 1$ and unstable if $\mathcal{R}_0 > 1$.

It is clear that the disease-free periodic
solution $(N,0,0)$ of system \eqref{SIRS3} is local asymptotically stable if all characteristic multipliers of periodic system \eqref{real_linear} are less than one, and
it is unstable if at least one of characteristic multipliers of periodic system \eqref{real_linear} is greater than one. By straightforward calculation,
we obtain that the characteristic multipliers of periodic system \eqref{real_linear} consist of $e^{-(d+\sigma)\omega}$ and the eigenvalues of the following
matrix
$$
\Phi_{F-V}(\omega)=e^{(F_2-V)\theta\omega}e^{(F_1-V)(1-\theta)\omega},
$$
where
\[
F_i-V=\left(
  \begin{array}{rr}
 \mu\beta_iN -(d+r_a) & \alpha\mu\beta_iN
\\
 (1-\mu)\beta_iN &  \alpha(1-\mu)\beta_iN -(d+r_s)
  \end{array}
 \right), \ \ i=1,2.
\]
Note that $e^{-(d+\sigma)\omega}<1$ because $d+\sigma>0$. Therefore, all characteristic multipliers of periodic system \eqref{real_linear} are less than one
if and only the largest eigenvalue of $\Phi_{F-V}(\omega)$, denoted by $\rho (\Phi_{F-V}(\omega))$, is less than one (i.e. $\rho (\Phi_{F-V}(\omega))<1$), and at least one of characteristic multipliers of periodic system \eqref{real_linear} is greater than one if and only if $\rho (\Phi_{F-V}(\omega))>1$, here $\rho (\Phi_{F-V}(\omega))$ is called
{\it the spectral radius} of matrix $\Phi_{F-V}(\omega)$.

\smallskip

On the other hand, it is easy to check that all assumptions (A2)-(A7) in \cite{Wang2008} are valid for system \eqref{real_linear} except the assumption (A1).
Using the notations in \cite{Wang2008}, we
define a matrix $V_{\varepsilon}=V-\varepsilon P$, here $P=\left( {\begin{array}{cc}
   1 & 1 \\       1 & 1      \end{array} } \right)$  and $\varepsilon$ is a very small positive number.
   Thus, $-V_{\varepsilon}$ is cooperative and irreducible for each $t\in \mathbb R$. Let $Y_{\varepsilon} (t,s)$ be the evolution operator of
   the linear system \eqref{v2eq} with $V$ replaced by $V_{\varepsilon}$. For some small $\varepsilon_0$, as $\varepsilon\in [0,\ \varepsilon_0)$
   we can define the linear operator
   ${\mathcal L}_{\varepsilon}$ by replacing $Y (t,s)$ in \eqref{operate} with $Y_{\varepsilon} (t,s)$ such that the operator
   ${\mathcal L}_{\varepsilon}$ is positive, continuous and compact on $\mathbb{C}_{\omega}$. Let $\mathcal{R}_0^{\varepsilon}:= \rho(\mathcal{L}_{\varepsilon})$
   for $\varepsilon\in [0,\ \varepsilon_0)$.

   By proof of Theorem \ref{existenceUni}, we know that the  solutions of the following system
   \begin{equation}\label{linearEq2}
   \frac{dx}{dt}=({\mathbb F}(t)-V_{\varepsilon})x
   \end{equation}
   are continuous  with respect to all parameters. Thus,
   \[
   \lim_{\varepsilon\to 0}\Phi_{F-V_{\varepsilon}}(\omega)=\Phi_{F-V}(\omega),
   \]
where $\Phi_{F-V_{\varepsilon}}(\omega)$ is the monodromy matrix of system \eqref{linearEq2}, and $\Phi_{F-V}(\omega)$ is the monodromy matrix of system \eqref{linearEq2}
as $\varepsilon=0$.

According to the continuity of the spectrum of matrices, we have
\[
\lim_{\varepsilon\to 0}\rho(\Phi_{F-V_{\varepsilon}}(\omega))=\rho(\Phi_{F-V}(\omega)).
\]
From Lemma \ref{operatorC}, we use the similar arguments in \cite{Wang2008} to the two linear operator
   ${\mathcal L}_{\varepsilon}$ and  ${\mathcal L}$, and obtain
   \[
   \lim_{\varepsilon\to 0}\mathcal{R}_0^{\varepsilon}=\mathcal{R}_0.
   \]

We now  easily follow the arguments in  \cite{Wang2008} to characterize $\mathcal{R}_0$. Let $W_{\lambda}(t, s), t\geq s$ be the fundamental
solution matrix of the following linear periodic system
\begin{equation*}\label{test}
\frac{dw}{dt}=\left(-V+\frac{\mathbb F(t)}{\lambda}\right)w,
\end{equation*}
where the parameter $\lambda\in (0, +\infty)$.  Consider an equation of $\lambda$
\begin{equation}\label{need}
\rho(W_{\lambda}(\omega, 0))=1.
\end{equation}
Then $\mathcal{R}_0$ can be calculated as follows.
\begin{theorem}\label{R0Characterize}
\begin{itemize}
\item[(i)] If equation \eqref{need} has a  solution $\lambda_0>0$, then $\lambda_0$ is an eigenvalue of $\mathcal{L}$, which implies that $\mathcal{R}_0>0$;
\item[(ii)] If $\mathcal{R}_0>0$, then $\lambda=\mathcal{R}_0$ is the only solution of equation \eqref{need};
\item[(iii)] $\mathcal{R}_0=0$ if and only if $\rho(W_{\lambda}(\omega, 0))<1$ for all positive $\lambda$.
\end{itemize}
\end{theorem}

Note that $\rho(W_1(\omega,0))=\rho(\Phi_{F-V}(\omega))$.
Using similar arguments in  \cite{Wang2008}, we can prove that the basic reproduction ratio (or number)
$\mathcal{R}_0$ can characterize  the threshold of disease invasion.
\begin{theorem}\label{threshold}
\begin{itemize}
\item[(i)]$\mathcal{R}_0>1$ if and only if $\rho (\Phi_{F-V}(\omega))>1$;
 \item[(ii)] $\mathcal{R}_0=1$ if and only if $\rho (\Phi_{F-V}(\omega))=1$;
 \item[(iii)] $\mathcal{R}_0<1$ if and only if $\rho (\Phi_{F-V}(\omega))<1$.
 \end{itemize}
 Hence, the disease-free periodic
solution $(N,0,0)$ of system \eqref{SIRS3} is local asymptotically stable if $\mathcal{R}_0<1$, and it is unstable if $\mathcal{R}_0>1$.
\end{theorem}

To save space, the proofs of the above theorems are omitted. From Theorem \ref{threshold}, we can see that $\mathcal{R}_0$ is a threshold parameter
for local stability of the disease-free periodic
solution $(N,0,0)$.   We next show $\mathcal{R}_0$ is also a threshold parameter for dynamics of system \eqref{SIRS3} in $\mathcal{D}_0$.
\medskip

\begin{theorem}\label{th-E02}
When $\mathcal{R}_{0}<1$, solutions
$(S(t),I_{a}(t),I_{s}(t))$ of  system  (\ref{SIRS3}) with initial points  in $\mathcal{D}_0$
satisfies
$$
\lim_{t \to +\infty}(S(t),I_{a}(t),I_{s}(t))=(N, 0, 0).
$$
And the disease-free periodic solution $(N,0,0)$ of system  (\ref{SIRS3}) is global asymptotically stable in $\mathcal{D}_0$.
\end{theorem}

\begin{proof}
In the invariant  pyramid $\mathcal{D}_0$ as shown in \eqref{D}, we consider a subsystem  by the last two equations of system  (\ref{SIRS3})
\begin{equation}
\begin{cases}
\label{compare-smaller}
\dot{I_{a}}(t)&=\mu\beta(t)S(I_{a}+\alpha I_{s})-(d+r_{a})I_{a}
\\
&\le \mu\beta(t) N(I_{a}+\alpha I_{s})-(d+r_{a})I_{a},
\\
\dot{I_{s}}(t)&= (1-\mu)\beta(t)S(I_{a}+\alpha I_{s})-(d+r_{s})I_{s}
\\
&\le (1-\mu)\beta(t)N(I_{a}+\alpha I_{s})-(d+r_{s})I_{s}.
\end{cases}
\end{equation}
Thus, the auxiliary  system of  \eqref{compare-smaller} is
 \begin{eqnarray}
\label{compare-bigger}
\begin{cases}
\dot{I_{a}}(t)=\mu\beta(t) N(I_{a}+\alpha I_{s})-(d+r_{a})I_{a},
\\
\dot{I_{s}}(t)=(1-\mu)\beta(t)N(I_{a}+\alpha I_{s})-(d+r_{s})I_{s},
\end{cases}
\end{eqnarray}
which is a periodic linear discontinuous system with period $\omega$. The periodic map associated with system \eqref{compare-bigger} is defined by
$\Phi_{F-V}(\omega)$, which is a linear continuous map.

When $\mathcal{R}_{0}<1$, we have  $\rho(\Phi_{F-V}(\omega))<1$, which implies that $(0,0)$ is a global asymptotically stable solution of system \eqref{compare-bigger}.

Note that systems \eqref{compare-smaller} and \eqref{compare-bigger} are cooperative.
Using the similar arguments in \cite{Smi1995}, we can prove the comparison principle holds. Hence,
$$
\lim_{t \to +\infty}(I_{a}(t),I_{s}(t))=(0, 0).
$$
So, for arbitrarily small constant $\varepsilon>0$, there exists $T>0$
 such that  $I_{a}(t)+\alpha I_{s}(t)<\varepsilon$ as $t>T$.
From the first equation  of  system  (\ref{SIRS3}),
\begin{equation*}
\begin{split}
\dot{S}&=dN-dS-\beta(t)S(I_{a}+\alpha I_{s})+\sigma(N-S-I_a-I_s)
\\
&> dN-dS-\beta_2S\varepsilon.
\end{split}
\label{SS1}
\end{equation*}
Therefore,  $\liminf_{t\rightarrow+\infty}S(t)\geqslant\frac{dN}{d+\beta_2\varepsilon}$. Let $\varepsilon\to 0$, we have
$$
 \liminf_{t\rightarrow+\infty}S(t)\geqslant N.
$$
On the other hand,   $S(t)\leqslant N$ in $\mathcal{D}_0$, which admits
$$
\lim_{t\rightarrow+\infty}S(t)=N.
$$
In summary, we have $\lim_{t \to +\infty}(S(t),I_{a}(t),I_{s}(t))=(N, 0, 0)$.
Moreover, from Theorem \ref{threshold} we know that $(N,0,0)$ of system \eqref{SIRS3} is global asymptotically stable.
\end{proof}


%
%
%
%
%


In the following, we show that the disease is
uniformly persistent when $\mathcal{R}_0 > 1$.


\begin{theorem}
If $\mathcal{R}_{0}>1$, $0<\mu<1$ and $0<\alpha\beta_1$, then 
there exists a  constant $\delta_{0}>0$ such that every solution
$(S(t),I_{a}(t),I_{s}(t))$ of  system  (\ref{SIRS3}) with initial value  in $\mathcal{D}_0$
satisfies
$$
\liminf_{t \to +\infty}I_{a}(t)\geqslant\delta_{0},\quad \liminf_{t \to +\infty}I_{s}(t)\geqslant\delta_{0}.
$$
\end{theorem}

\begin{proof}
Since system \eqref{SIRS3} is $\omega$-periodic with respect to $t$ in $\mathbb{R}_+\times \mathcal{D}_0$, it suffices to investigate the dynamics of its associated period map $\mathcal{P}$ defined by \eqref{poincaremap} on $\mathcal{D}_0$ for the dynamics of system \eqref{SIRS3}, where the map $\mathcal{P}$ is continuous. Clearly, $\mathcal{P}(\mathcal{D}_0)\subset \mathcal{D}_0$.
Define
$$
X_{0}=\{(S,I_a,I_s)\in\mathcal{D}_0: I_a>0,I_s>0\},\ \partial{X_{0}}=\mathcal{D}_0 \backslash X_{0}.
$$

Set
$$
M_{\partial}=\{P_0\in \partial X_{0} :\mathcal{P}^k(P_0)\in\partial X_{0},  \forall k\ge 0 \},
$$
which is a positive invariant set of $\mathcal{P}$ in $\partial X_{0}$. We claim
\begin{equation}
\label{M-partial}
M_{\partial}=\{(S,0,0):0\leqslant S\leqslant N\}.
\end{equation}
In fact,  $\{(S,0,0):0\leqslant S\leqslant N\}\subset M_{\partial}$ by \eqref{poincaremap}.
On the other hand, for any
$P_0\in\partial{X_{0}}\setminus \{(S,0,0):0\leqslant S\leqslant N\}$, that is either
 $I_{a0}=0, I_{s0}>0, S_0\ge0$ or $I_{a0}>0,I_{s0}=0, S_0\ge0$.
In the case  $I_{a0}=0, I_{s0}>0, S_0>0$ (resp. $I_{a0}>0, I_{s0}=0,  S_0>0$), we calculate by the last two equations
of system \eqref{SIRS3} and obtain that
\begin{eqnarray*}
I_{a}'(0)=\mu\alpha\beta(0) S(0)I_{s}(0)>0\ ({\rm resp.} \ ~I_{s}'(0)=(1-\mu)\beta(0) S(0)I_{a}(0)>0),
\end{eqnarray*}
if $0<\mu<1$ and $0<\alpha\beta_1$. This implies that $\mathcal{P}^{k_0}(P_0)\not\in\partial{X_{0}}\setminus \{(S,0,0):0\leqslant S\leqslant N\}$ for some $k_0\ge0$ since the subsystem by the last two equations
of system \eqref{SIRS3} is cooperative.  If $S(0)=0, I_{a0}=0, I_{s0}>0$ (or $S(0)=0, I_{a0}>0,I_{s0}=0$), then
$S'(0)=(d+\sigma)N-\sigma I_s(0)>0$ (or $S'(0)=(d+\sigma)N-\sigma I_a(0)>0$),
which leads that $\mathcal{P}^{k_1}(P_0)\not\in\partial{X_{0}}\setminus \{(S,0,0):0\leqslant S\leqslant N\}$ for some $k_1\ge 0$.
Therefore, \eqref{M-partial} is proved and $M_{\partial}$ is the maximal compact invariant set of $\mathcal{P}$ in $\partial X_{0}$.

Note that $E_0(N, 0, 0)$ is the unique fixed point  of $\mathcal{P}$ in $M_{\partial}$ and it is an attractor of $\mathcal{P}$ in $M_{\partial}$
by the first equation of \eqref{SIRS3}.
Since $\mathcal{R}_{0}>1$, the stable set $W^s(E_0)$ of $E_0$ satisfies that $W^s(E_0) \cap X_0 =\emptyset$.

Applying  \cite[Theorem 1.3.1]{Zhao2003}, we obtain that $\mathcal{P}$ is uniformly persistence with respect to $(X_0, \partial X_0)$.
Moreover, from \cite[Theorem 3.1.1]{Zhao2003}, it can see that the conclusion of this theorem is true. The proof is completed.
\end{proof}


\medskip


\section{Global dynamics of system \eqref{SIRS3} without seasonal force}

In this section, we study  the effects of asymptomatic infection on  dynamics of system \eqref{SIRS3} if there are not seasonal factors, that is,
 $\beta_1=\beta_2=\beta$. Then system \eqref{SIRS3} becomes
\begin{equation}\label{SIR-dim3}
\begin{cases}
\dot{S}=(d+\sigma)(N-S)-\beta S(I_a+\alpha I_s)-\sigma (I_a+I_s),\\
\dot{I_a}=\mu\beta S(I_a+\alpha I_s)-(d+r_a)I_a,\\
\dot{I_s}=(1-\mu)\beta S(I_a+\alpha I_s)-(d+r_s)I_s
\end{cases}
\end{equation}
in the domain $\mathbb{R}_+^3$.

By the formula \eqref{R_0}, we let $\beta_1=\beta_2$ and obtain the basic reproduction number
 $\mathcal{R}_0$ of system \eqref{SIR-dim3} as follows.
 \begin{eqnarray}\label{cons-R0}
\mathcal{R}_0=\beta N \left(  \frac{\mu}{d+r_{a}}+\frac{\alpha(1-\mu)}{d+r_{s}} \right),
\end{eqnarray}
 which is consistent with the number calculated using the approach of basic reproduction number in  \cite{Die1990} and  \cite{Van2002}.

 From the expression \eqref{cons-R0}, we can see that there is still the risks of infectious disease outbreaks due to the existence of asymptomatic infection
 even if all symptomatic infective individuals has been quarantined, that is, $\alpha=0$. This provides an intuitive basis for
  understanding that the asymptomatic infective individuals promote  the evolution of epidemic.

In the following we study dynamics of system \eqref{SIR-dim3}.
By a straightforward calculation, we obtain  the  existence of equilibrium for system \eqref{SIR-dim3}.

\begin{lemma}\label{existen}
\label{L-equils}(Existence of equilibrium) System (\ref{SIR-dim3}) has the following equilibria in $\mathbb{R}_+^3$.
\begin{itemize}
\item[(i)]
If $\mathcal{R}_0\le 1$, then system (\ref{SIR-dim3})  has a unique  equilibrium, which is the disease-free equilibrium $E_0(N,0,0)$.
\item[(ii)] If $\mathcal{R}_0>1$ and $0<\mu<1$, then system (\ref{SIR-dim3})  has two equilibria: the disease-free equilibrium $E_0(N,0,0)$ and the endemic equilibrium $E_1(S^*, I_a^*, I_s^*)$    in the interior of  $\mathcal{D}_0$,
where
$S^* = \frac{N}{\mathcal{R}_0},
I_a^* = \frac{\mu(d+\sigma) (d+r_s)N}
{(d+r_a)(d+r_s)+\sigma(d+\mu r_s)+\sigma r_a(1-\mu)} (1-\frac{1}{\mathcal{R}_0}),
I_s^* =\frac{(1-\mu)(d+r_a)}{\mu(d+r_s)}I_a^*.$
\item[(iii)] If $\mathcal{R}_0>1$ and $\mu=0$, then system (\ref{SIR-dim3}) has two equilibria: the disease-free equilibrium $E_0(N,0,0)$  and the asymptomatic-free equilibrium $E_2 (S_2^*,0,I_{s2}^*)$, where
$S_2^*=  \frac{N}{\mathcal{R}_0},
I_{s2}^*= \frac{d+\sigma}{d+\sigma+r_s}N(1-\frac{1}{\mathcal{R}_0}).$
\item[(iv)] If $\mathcal{R}_0>1$ and $\mu=1$, the system (\ref{SIR-dim3}) has two equilibria: the disease-free equilibrium $E_0(N,0,0)$   and the symptomatic-free equilibrium $E_3 (S_3^*,I_{a3}^*,0)$, where
$S_3^*=  \frac{N}{\mathcal{R}_0},
I_{a3}^*= \frac{d+\sigma}{d+\sigma+r_a}N(1-\frac{1}{\mathcal{R}_0}).$
\end{itemize}
\end{lemma}

\medskip

We now discuss the local stability and topological classification of these equilibria in $\mathbb{R}_+^3$, respectively. We first study the disease-free equilibrium $E_0(N, 0, 0)$
and have the following lemma.

\begin{lemma}
\label{localstablity}
The disease-free equilibrium $E_0(N, 0, 0)$ of  system (\ref{SIR-dim3}) in $\mathbb{R}_+^3$  is  asymptotically stable if $\mathcal{R}_0<1$;
$E_0(N, 0, 0)$ is a saddle-node with one dimensional center manifold and two dimensional stable manifold if
$\mathcal{R}_0=1$; and $E_0(N, 0, 0)$ is a saddle with two dimensional stable manifold and  one dimensional unstable manifold if $\mathcal{R}_0>1$.
\end{lemma}

\begin{proof}

A routine computation shows that the characteristic
polynomial of  system  (\ref{SIR-dim3}) at $E_0$ is
\begin{eqnarray}\label{ChEq}
f_1(\lambda) = (\lambda+d+\sigma)(\lambda^{2}-a_1\lambda+a_0),
\end{eqnarray}
where $a_0=(d+r_{a})(d+r_{s})(1-\mathcal{R}_0),$
$$
a_1=(d+r_{a})(\beta N\frac{\mu}{d+r_{a}}-1)+(d+r_{s})(\alpha\beta N\frac{1-\mu}{d+r_{s}}-1) .
$$
It is clear that
$-(d+\sigma)<0$ is always one root of \eqref{ChEq}.
We divide three cases: $\mathcal{R}_0<1$, $\mathcal{R}_0=1$ and $\mathcal{R}_0>1$ to discuss the other roots of \eqref{ChEq}.

If $\mathcal{R}_0<1$, then $a_1<0$ and $a_0>0$ by  $\beta N\frac{\mu}{d+r_{a}}<\mathcal{R}_0$ and $\beta N\frac{\alpha(1-\mu)}{d+r_{s}}<\mathcal{R}_0$.
Thus,  three roots of \eqref{ChEq} have negative real parts, which leads to the local asymptotically stable of the disease-free equilibrium $E_0$.


If $\mathcal{R}_0=1$, then $a_0=0$ and $a_1<0$. Hence, the characteristic equation $f_1(\lambda) =0$ has three roots: $\lambda_1=-(d+\sigma)<0$, $\lambda_2=a_1<0$ and $\lambda_3=0$.
For calculating the associated eigenvectors $v_i$ of $\lambda_i$, $i=1,2,3$,  we consider $J(E_0)$ with respect to $\mu$ in three cases: (i) $0<\mu<1$, (ii) $\mu=0$ and (iii) $\mu=1$, and we can obtain that $E_0$ is a saddle-node with one dimensional center manifold and two dimensional stable manifold by tedious calculations of normal form.

Summarized the above analysis, we complete proof of this lemma.
\end{proof}


From lemma \ref{existen} and lemma \ref{localstablity}, we can see that system \eqref{SIR-dim3} undergoes  saddle-node bifurcation in a small neighborhood of $E_0(N,0,0)$
as $\mathcal{R}_0$ increases passing through $\mathcal{R}_0=1.$





About the endemic equilibria, we have the following local stability.

\begin{lemma}
\label{L-E1}
The  endemic equilibrium $E_1(S^*, I_a^*, I_s^*)$ of  system (\ref{SIR-dim3})  is   asymptotically stable  if $\mathcal{R}_0>1$ and $0<\mu<1$; the  asymptomatic-free equilibrium $E_2 (S_2^*,0,I_{s2}^*)$  of  system (\ref{SIR-dim3})  is   asymptotically stable  if $\mathcal{R}_0>1$ and $\mu=0$; and the  symptomatic-free equilibrium $E_3 (S_3^*,I_{a3}^*,0)$  of  system (\ref{SIR-dim3})  is  asymptotically stable  if $\mathcal{R}_0>1$ and $\mu=1$.
\end{lemma}

\begin{proof}
Either $\mu=0$ or $\mu=1$, it is easy to compute the eigenvalues of the Jacobian matrix of  system  (\ref{SIR-dim3}) at $E_2$ or $E_3$, respectively, and
find that all eigenvalues have the negative real parts. Hence, $E_2 (S_2^*,0,I_{s2}^*)$ or $E_3 (S_3^*,I_{a3}^*,0)$ is  asymptotically stable if $\mathcal{R}_0>1$, respectively.

After here we only
prove that $E_1(S^*, I_a^*, I_s^*)$  is   asymptotically stable  if $\mathcal{R}_0>1$ and $0<\mu<1$.
To make the calculation easier, we use the variables change
\begin{equation*}\label{change1}
S=  \frac{ (d+r_s)}{\mu\beta}  \hat{S}, ~I_a= \frac{ (d+r_s)}{\beta} \hat{I}_a, ~~I_s= \frac{ (d+r_s)}{\beta}\hat{I}_s, ~dt= \frac{d\tau}{(d+r_s)},
\end{equation*}
which reduces system (\ref{SIR-dim3}) into the following system,
\begin{equation}\label{SIR-1}
\begin{cases}
\frac{dS}{d\tau}=N_1-d_1S-\sigma_1 I_a-\sigma_1I_s - S(I_a+\alpha I_s),\\
\frac{dI_a}{d\tau}=-r I_a + S(I_a+\alpha I_s),\\
\frac{dI_s}{d\tau}=-I_s +\mu_1 S(I_a+\alpha I_s),
\end{cases}
\end{equation}
where
\begin{eqnarray*}\label{Pchange1}
\begin{split}
N_1 &=N (d+\sigma) \mu \beta/(d+r_s)^2 , ~d_1=(d+\sigma)/(d+r_s),
\\
\sigma_1 &=\sigma\mu/(d+r_s), ~r=(d+r_a)/(d+r_s), ~\mu_1=(1-\mu)/\mu
\end{split}
\end{eqnarray*}
and for simplicity we denote $\hat{S}, \hat{I}_a, \hat{I}_s$ by $S, I_a, I_s$ respectively.

  When $\mathcal{R}_0>1$, the disease-free equilibrium $E_0(N, 0, 0)$ and endemic equilibrium $E_1(S^*, I_a^*, I_s^*)$ of  system (\ref{SIR-dim3}) are transformed into the disease-free equilibrium $\hat{E}_0(N_1/d_1, 0, 0)$ and endemic equilibrium $\hat{E}_1(\hat{S}^*, \hat{I}_a^*, \hat{I}_s^*)$ of  system (\ref{SIR-1}) respectively, where
\begin{eqnarray*}
\begin{split}
\hat{S}^* =\frac{N_1/d_1}{\hat{R}_0}, ~\hat{I}_a^* =  \frac{N_1}{\sigma_1+r \sigma_1 \mu_1+r}(1-\frac{1}{\hat{R}_0}),
~\hat{I}_s^*  = \mu_1 r I_a^*.
\end{split}
\end{eqnarray*}
Notice that  $\hat{R}_0:=\frac{N_1}{d_1}(\frac{1}{r}+\alpha\mu_1)>1$  if and only if $\mathcal{R}_0>1$.

The characteristic
equation of system (\ref{SIR-1}) at $\hat{E}_1$ is
\begin{eqnarray*}\label{Ch-E1}
f_2(\lambda) ={\rm det} (\lambda I-J(\hat{E}_1) ) = \lambda^3+ \xi_2 \lambda^2 + \xi_1 \lambda +\xi_0,
\end{eqnarray*}
where
\begin{eqnarray*}
\begin{split}
\xi_2 &=\{\sigma_1+r \sigma_1 \mu_1+r+r^2 \mu_1 \alpha \sigma_1+r^3 \mu_1^2 \alpha \sigma_1+r^3 \mu_1 \alpha+d_1 \sigma_1 \mu_1 r \alpha+d_1 \sigma_1 \mu_1^2 r^2 \alpha+N_1
\\
~~ & +d_1 \sigma_1+d_1 r \sigma_1 \mu_1+2 N_1 \mu_1 r \alpha+r^2 \mu_1^2 \alpha^2 N_1\}/\{(\sigma_1+r \sigma_1 \mu_1+r)(r \mu_1 \alpha+1)\},
\\
\xi_1 &= d_1 (1+r^2 \mu_1 \alpha)/(r \mu_1 \alpha+1) +(\sigma_1 \mu_1+1+r+\sigma_1) (r \mu_1 \alpha+1) \hat{I}_a^*,
\\
\xi_0 &=N_1\mu_1 r\alpha+N_1-rd_1=rd_1(\hat{R}_0-1).
\end{split}
\end{eqnarray*}
It can be seen that  all coefficients $\xi_j$ of polynomial $f_2(\lambda)$ are positive if $\hat{R}_0>1$, where $j=0, 1, 2$.
Moreover,  we claim that $\xi_2\xi_1-\xi_0>0$. In fact,
\begin{eqnarray*}
\begin{split}
\xi_2\xi_1-\xi_0 =c_0+c_1 \hat{I}_a^* +c_2 (\hat{I}_a^*)^2,
\end{split}
\end{eqnarray*}
where
\begin{eqnarray*}
\begin{split}
c_0= & ~ \frac{d_1(1+r^2\mu_1 \alpha) (r^2 \mu_1 \alpha+ d_1\mu_1 r\alpha+1+d_1)}{(r \mu_1 \alpha+1)^2},
\\
c_1= & ~ d_1 \mu_1^2 r \alpha \sigma_1+r^3 \mu_1 \alpha+r^2 \mu_1 \alpha \sigma_1+2 d_1 r^2 \mu_1 \alpha+d_1 \sigma_1 \mu_1 r \alpha+\sigma_1 \mu_1
\\
&~+d_1 \sigma_1 \mu_1+1+2 d_1+d_1 \sigma_1    +r(d_1-\sigma_1 \mu_1) + \mu_1 r \alpha(d_1-\sigma_1),
\\
c_2= & ~ (r \mu_1 \alpha+1)^2 (\sigma_1\mu_1+1+r+\sigma_1).
\end{split}
\end{eqnarray*}
It is easy to see that $c_0>0$ and $c_2>0$. Note that $ d_1-\sigma_1 \mu_1=\frac{d+\sigma \mu}{d+r_s}>0$ and $d_1-\sigma_1=\frac{d+\sigma (1- \mu)}{d+r_s}>0$ since $0<\mu<1$.  This implies that $c_1>0$. Moreover, $\hat{I}_a^*>0$ yields that $\xi_2\xi_1-\xi_0>0$ and what we claimed is proved.

By the Routh-Hurwitz Criterion, we know that all eigenvalues of the characteristic
polynomial $f_2(\lambda)$ have negative real parts. Thus, endemic equilibrium $\hat{E}_1$ of  system (\ref{SIR-1})
is asymptotically stable. This leads that endemic equilibrium $E_1(S^*, I_a^*, I_s^*)$ of  system (\ref{SIR-dim3}) is also asymptotically stable.
\end{proof}

From lemma \ref{localstablity} and lemma \ref{L-E1}, we can see that $\mathcal{R}_0$ is the threshold quantity
of local dynamics of system \eqref{SIR-dim3} in $\mathbb{R}^3_+$.
 By Theorem \ref{existenceUni}, we only need to consider  system \eqref{SIR-dim3}
 for its global dynamics in $\mathcal{D}_0$.
The following theorems will show that $\mathcal{R}_0$ is also the threshold quantity of global dynamics of system \eqref{SIR-dim3} in $\mathcal{D}_0$.

\begin{theorem}
\label{globalE0}
If $\mathcal{R}_0\le 1$, then the disease-free equilibrium $E_0(N, 0, 0)$ of system (\ref{SIR-dim3}) is global asymptotically stable in $\mathcal{D}_0$.
\end{theorem}

The proof of this theorem can be finished by
constructing a Liapunov function
 \begin{equation*}
 \label{Lia1}
L(S, I_a,I_s)  =I_a(t) +\frac{d+r_a}{d+r_s}\alpha I_s(t)
\end{equation*}
in $\mathcal{D}_0$. For saving space, we bypass it.


\begin{theorem}
\label{globalE1}
If $\mathcal{R}_0> 1$ and $\mu=0$ (resp. $\mu=1$), then    $E_2(S_2^*,0,I_{s2}^*)$
(resp. $E_3(S_3^*,I_{a3}^*,0)$) attracts all orbits of system (\ref{SIR-dim3}) in $\mathcal{D}_0$ except both $E_0(N, 0, 0)$
and a positive orbit $\gamma$ in its two dimensional stable manifold,
where
$$
\gamma=\{(S,I_a,I_s)\in \mathcal{D}_0:\ I_a=0,\ I_s=0, \ 0<S<N \} .
$$
\end{theorem}

\begin{proof}
We first prove the case that $\mathcal{R}_0> 1$ and $\mu=0$.  When $\mu=0$, system (\ref{SIR-dim3}) becomes
\begin{equation}\label{mu0}
\begin{cases}
\dot{S}=(d+\sigma)(N-S)-\beta S(I_a+\alpha I_s)-\sigma (I_a+I_s),\\
\dot{I_a}=-(d+r_a)I_a,\\
\dot{I_s}=\beta S(I_a+\alpha I_s)-(d+r_s)I_s.
\end{cases}
\end{equation}
It is clear that $\lim_{t\to +\infty}I_a(t)=0$. Hence, the limit system of system \eqref{mu0} in $\mathcal{D}_0$  is
\begin{equation}\label{mu0limit}
\begin{cases}
\dot{S}=(d+\sigma)(N-S)-\alpha\beta S I_s-\sigma I_s,\\
\dot{I_s}=\alpha\beta S I_s-(d+r_s)I_s
\end{cases}
\end{equation}
in $\mathcal{D}_{1}=\{(S,I_s):\ 0\le S\le N,\ 0\le I_s\le N\}$, which has two equilibria: $(N,0)$ and $(S_2^*,I_{s2}^*)$.
Equilibrium $(N,0)$ is a saddle and $(S_2^*,I_{s2}^*)$ is locally asymptotically stable if $\mathcal{R}_0>1$.

In the following we  prove that $(S_2^*,I_{s2}^*)$ attracts all orbits of system \eqref{mu0limit} in $\mathcal{D}_{1}$ except both $(N,0)$ and its one dimensional stable manifold.

Let $x=S+\frac{\sigma}{\alpha\beta}$ and $y=I_s$. Then system \eqref{mu0limit} becomes
\begin{equation}\label{mu0xy}
\begin{cases}
\dot{x}=(d+\sigma)(N+\frac{\sigma}{\alpha\beta})-(d+\sigma)x-\alpha\beta xy,\\
\dot{y}=\alpha\beta xy-(d+r_s+\sigma)y.
\end{cases}
\end{equation}
Hence, $(x_0,y_0)=(S_2^*+\frac{\sigma}{\alpha\beta},I_{s2}^*)$ is the unique positive equilibrium of system \eqref{mu0xy} if $\mathcal{R}_0>1$.
Consider the Liapunov function of system \eqref{mu0xy}
$$
V(x,y)=\frac{1}{2}(x-x_0)^2+x_0\left(y-y_0-y_0\ln\frac{y}{y_0}\right)
$$
in $\tilde{\mathcal{D}}_{1}=\{(x,y):\ \frac{\sigma}{\alpha\beta}\le x\le N+\frac{\sigma}{\alpha\beta},\ 0\le y\le N\}$.
It is clear that $V(x,y)\ge 0$ and $V(x,y)=0$ if and only if $x=x_0$ and $y=y_0$ in $\tilde{\mathcal{D}}_{1}$. And
$$
\frac{dV(x(t),y(t))}{dt}|_{\eqref{mu0xy}}=-(x-x_0)^2(\alpha\beta y+d+\sigma)\le 0
$$
in $\tilde{\mathcal{D}}_{1}$.

 By  LaSalle's Invariance Principle, we know that $(x_0,y_0)$ attracts all orbits of system \eqref{mu0xy}
in $\tilde{\mathcal{D}}_{1}$ except both equilibrium $(N+\frac{\sigma}{\alpha\beta},0)$ and its one dimensional stable manifold $\{(x,y):\ y=0, 0<x<N+\frac{\sigma}{\alpha\beta}\}$. This leads to the conclusion, $E_2(S_2^*,0,I_{s2}^*)$
 attracts all orbits of system (\ref{SIR-dim3}) in $\mathcal{D}_0$ except both $E_0(N, 0, 0)$ and a positive orbit $\gamma $ if $\mathcal{R}_0> 1$ and $\mu=0$.

 Using the similar arguments, we can prove that $E_3(S_3^*,I_{a3}^*,0)$ attracts all orbits of system (\ref{SIR-dim3}) in $\mathcal{D}_0$
 except both $E_0(N, 0, 0)$ and a positive orbit $\gamma $ if $\mathcal{R}_0> 1$ and $\mu=1$.
\end{proof}

 \begin{theorem}
\label{th-globE1}
If $\mathcal{R}_0> 1$, $0<\mu<1$ and $r_a=r_s$,  then the endemic equilibrium $E_1(S^*, I_a^*, I_s^*)$
attracts all orbits of system (\ref{SIR-dim3}) in $\mathcal{D}_0$ except  $E_0(N, 0, 0)$.
\end{theorem}

\begin{proof}
Let
$I=I_a+\alpha I_s, ~~~N_1=S+I_a+I_s$.
Then under the assumption $r_a=r_s=r$, system (\ref{SIR-dim3}) in $\mathcal{D}_0$ can be written as
\begin{equation}\label{SIN}
\begin{cases}
\dot{S}=(d+\sigma)N -\sigma N_1 -dS -\beta SI,\\
\dot{I}=\tilde{\mu}SI-(d+r)I,\\
\dot{N_1}=(d+\sigma)N -(d+r+\sigma) N_1+rS
\end{cases}
\end{equation}
in $\tilde{\mathcal{D}}_0:=\{(S, I, N_1)|\; S\ge 0,  I\ge 0,   ~N\ge N_1\ge 0 \}$,
where $\tilde{\mu}=(\mu+\alpha(1-\mu))\beta$.

Thus, equilibrium $E_1(S^*, I_a^*, I_s^*)$
 of system (\ref{SIR-dim3}) becomes equilibrium $\tilde{E}_1(S^*, I^*, N_1^*)$  of system (\ref{SIN}) and  $\tilde{E}_1$ is locally asymptotically stable,
 where $I^*=I_a^*+\alpha I_s^*, N_1^*=S^* + I_a^* + I_s^*$.

Applying a typical approach of Liapunov functional,
we define
\begin{equation*}\label{gx}
g(x) = x- 1- \ln x,
\end{equation*}
and construct  a Liapunov functional of system (\ref{SIN})
\begin{equation*}\label{V1}
V_1(S, I, N_1)=\frac{\nu_1}{2} (S-S^*)^2+\nu_2 I^*g(\frac{I}{I^*}) + \frac{\nu_3}{2} (N_1-N_1^*)^2,
\end{equation*}
where arbitrary constant $\nu_1>0$, $\nu_2=\nu_1\beta S^*/\tilde{\mu}$ and $\nu_3=\nu_1\sigma/r$.
Note that $g (x)\ge g (1) = 0$ for all $x > 0$   and the global minimum $g (x) = 0$ is attained if and only if $x = 1$.
Thus, $V_1(S, I, N_1)\ge 0$ and $V_1(S, I, N_1)=0$ if and only if $S=S^*$, $I=I^*$ and $N_1=N_1^*$ in
$\tilde{\mathcal{D}}_0$.

The derivative of $V_1$ along the trajectories of system  \eqref{SIN} is
\begin{equation*}\label{dV1}
\begin{split}
\frac{dV_1(S,I, N_1)}{dt}=& -\nu_1 d {S^*}^2(x-1)^2 -\nu_3(d+r+\sigma) {N_1^*}^2(z-1)^2
\\
& - \nu_1\beta  {S^*}^2I^*y(x-1)^2\le 0,
\end{split}
\end{equation*}
 where
$x=\frac{S}{S^*}, ~~y=\frac{I}{I^*}, ~~ z=\frac{N_1}{N_1^*}$.

Note that the only compact invariant subset of the set $\{(S,I, N_1):\ \frac{dV_1(S,I, N_1)}{dt}=0\}$ is the singleton $\tilde{E}_1(S^*, I^*, N_1^*)$
in $\tilde{\mathcal{D}}_0$. Consequently, we can conclude that $E_1(S^*, I_a^*, I_s^*)$ is globally asymptotically stable
and attracts all orbits of system (\ref{SIR-dim3}) in $\mathcal{D}_0$ except  $E_0(N, 0, 0)$.
\end{proof}

From Theorem \eqref{th-globE1} and  the continuity of solutions with respect to parameters $r_a$ and $r_s$, we obtain
 the following results.
\begin{theorem}
\label{th-globE1-2}
If $\mathcal{R}_0> 1$ and $0<\mu<1$,  then the endemic equilibrium $E_1(S^*, I_a^*, I_s^*)$
is globally asymptotical stable in  the interior of   $\mathcal{D}_0$  for  $ 0<|r_s-r_a|\ll 1$.
\end{theorem}

\section{Discussion}

In this model, we divide the period of the disease transmission into two seasons. In fact, it can be divided into  $n$ seasons for any given $n\in\mathbb{Z}_+$.
Compared with continuous periodic systems, our piecewise  continuous periodic model can provide a straightforward method to evaluate the basic reproduction
number $\mathcal{R}_0$, that is to calculate the spectral radius of matrix $\Phi_{F-V}(\omega)=e^{(F_2-V)\theta\omega}e^{(F_1-V)(1-\theta)\omega}$.
 It is shown that the length of the season, the transmission rate
and the existence of asymptomatic infective affect the basic reproduction number  $\mathcal{R}_0$, and there is still the risks of infectious disease outbreaks due to the existence of asymptomatic infection
 even if all symptomatic infective individuals has been quarantined, that is, $\alpha=0$. This provides an intuitive basis for
  understanding that the  asymptomatic infective individuals and the disease seasonal transmission promote  the evolution of epidemic.
  And theoretical dynamics of the model allow us to predictions of outcomes of control strategies during the course of the epidemic.


\end{document}